\documentclass[a4paper,12pt]{article}

\usepackage[utf8]{inputenc}

\usepackage{lmodern}
\usepackage[T1]{fontenc}

\usepackage[top=1in, bottom=1.25in, left=1.25in, right=1.25in]{geometry}
\usepackage{amsmath,amssymb,amsthm}

\usepackage{tikz}
\usepackage{xcolor}
\usepackage{array}
\usepackage{enumerate}
\usepackage{graphicx}
\usepackage{float}
\usepackage{caption}
\usepackage{subcaption}

\usepackage[colorlinks=true,linkcolor=red,citecolor=blue,urlcolor=red]{hyperref}

\newtheorem{thm}{Theorem}
\newtheorem{cor}[thm]{Corollary}
\newtheorem{lem}[thm]{Lemma}
\newtheorem{prop}[thm]{Proposition}
\newtheorem*{thm*}{Theorem}

\theoremstyle{definition}

\theoremstyle{remark}

\newtheorem*{remarks}{Remarks}

\newcommand{\overbar}[1]{\mkern 1.5mu\overline{\mkern-1.5mu#1\mkern-1.5mu}\mkern 1.5mu}

\author{Sel\c{c}uk Kayacan\thanks{This work is supported by the T{\"U}B\.{I}TAK 2214/A Study Abroad Fellowship Program: 1059B141401085.}}
\title{$K_{3,3}$-free Intersection Graphs of Finite Groups\thanks{I would like to thank to the anonymous referee for the meticulous reading and the shorter arguments he/she provides which greatly enhances the exposition. This work is completed while I was a visitor in Graz University of Technology. By this occasion I would like to thank to Math C Department of TUGraz for their hospitality.}}
\date{}

\begin{document}

\maketitle

\small
\begin{center}
Department of Mathematics, Istanbul Technical University,
34469 Maslak, Istanbul, Turkey. \\
{\it e-mail:} \href{mailto:skayacan@itu.edu.tr}{skayacan@itu.edu.tr}
\end{center}

\begin{abstract}
  The intersection graph of a group $G$ is an undirected graph without
  loops and multiple edges defined as follows: the vertex set is the
  set of all proper non-trivial subgroups of $G$, and there is an edge
  between two distinct vertices $H$ and $K$ if and only if
  $H\cap K \neq 1$ where $1$ denotes the trivial subgroup of $G$. In
  this paper we classify all finite groups whose intersection graphs
  are $K_{3,3}$-free. 


\end{abstract}

\subsection*{Introduction}
\label{sec:introduction}

Let $\overbar{\Gamma}$ and $\Gamma$ be two graphs. We say that $\overbar{\Gamma}$ is $\Gamma$-free if there is no subgraph of $\overbar{\Gamma}$ which is isomorphic to $\Gamma$, i.e. $\overbar{\Gamma}$ does not contain $\Gamma$ as a subgraph. Let $G$ be a group. We denote by $\Gamma(G)$ the intersection graph (of subgroups) of $G$. For simplicity we say that $G$ is $\Gamma$-free whenever its intersection graph is $\Gamma$-free.

In a recent work \cite{RajkumarDevi}, Rajkumar and Devi classified finite groups whose intersection graphs does not contain one of $K_5$, $K_4$, $C_5$, $C_4$, $P_4$, $P_3$, $P_2$, $K_{1,3}$, $K_{2,3}$ or $K_{1,4}$ as a subgraph. Here we present the classification of finite $K_{3,3}$-free groups. Our main result is 

\begin{thm*} A finite group is $K_{3,3}$-free if and only if it is isomorphic to the one of the following groups:
\begin{enumerate}
\itemsep-1em 
\item $\mathbb{Z}_{pqr}$, $\mathbb{Z}_{p^2q}$, $\mathbb{Z}_{pq}$, $\mathbb{Z}_{p^{i}}$, where $p,q,r$ are distinct primes and $0\leq i\leq 6$.\\
\item $\mathbb{Z}_{4}\times \mathbb{Z}_{2}$, $\mathbb{Z}_{p}\times \mathbb{Z}_{p}$, $\mathbb{Z}_{2}\times \mathbb{Z}_{2}\times \mathbb{Z}_{p}{\ }(p\neq 2)$, where $p$ is a prime.\\
\item The dihedral group $D_8$ of order $8$, the quaternion group $Q_8$ of order $8$.\\
\item The semidirect products $\mathbb{Z}_q \rtimes \mathbb{Z}_{p^2}$ with $p^2\bigm{|}q-1$, $(\mathbb{Z}_p \times \mathbb{Z}_p) \rtimes \mathbb{Z}_q$ with $q\bigm{|}p+1$, where $p,q$ are distinct primes.\\
\item The semidirect product $(\mathbb{Z}_p \times \mathbb{Z}_p) \rtimes \mathbb{Z}_{q^2}$ with $q^2\bigm{|}p+1$, where $p,q$ are distinct primes.\\
\item The semidirect product $\mathbb{Z}_r \rtimes \mathbb{Z}_{pq}$ with $pq\bigm{|}r-1$, where $p,q,r$ are distinct primes.\\
\item The semidirect product $\mathbb{Z}_p\rtimes \mathbb{Z}_q$ with $q\bigm{|}p-1$, where $p>q$ are distinct primes.\\
\item $\mathbb{Z}_{p^3}\times \mathbb{Z}_q$, $\mathbb{Z}_9\times\mathbb{Z}_3$, $(\mathbb{Z}_3\times\mathbb{Z}_3)\rtimes\mathbb{Z}_3$, $\mathbb{Z}_9\rtimes\mathbb{Z}_3$, $\mathbb{Z}_3\rtimes\mathbb{Z}_4$, and $D_{18}$, where $p,q$ are distinct primes.\\
\item The semidirect product $\mathbb{Z}_q \rtimes \mathbb{Z}_{p^3}$ with $p^3\bigm{|}q-1$, where $p,q$ are distinct primes.
\end{enumerate}
\end{thm*}

In \cite{KayacanYaraneri2015a}, Kayacan and Yaraneri classified all finite groups whose intersection graphs are planar (planar groups in short). By Kuratowski's theorem a graph is planar if and only if it does not contain a subgraph that is a subdivision of $K_5$ or $K_{3,3}$. Therefore the above list contains all finite groups with planar intersection graphs. Actually, the first seven items lists the planar groups (except $\mathbb{Z}_{p^6}$) and the last two items presents the $K_{3,3}$-free groups containing a $K_5$ as a subgraph in their intersection graphs. Our methods are elementary. Occasionally, we used some improved versions of the arguments that are presented in \cite{KayacanYaraneri2015a}. Notation is standard and should be clear to the reader. Below we present some of the results we used in the text.

\begin{remarks}
\mbox{}
  \begin{enumerate} 
  \item \label{rem1} (see \cite[p.~30]{Rotman1995}) (Product Formula)
    $|XY||X\cap Y|=|X||Y|$ for any two subgroups $X$ and $Y$ of a finite group.
  \item \label{rem2} (see \cite[p.~30]{Berkovich2008}) (Sylow Theorems)
    \begin{enumerate}
    \item Let $G$ be a group of order $p^n$ and $k < n$. Then the number of subgroups of order $p^k$ in $G$ is $\equiv 1 \pmod{p}$.
    \item Let $G$ be a group of order $p^ns$, $p\nmid s$, $k\leq n$. Then the number of subgroups of order $p^k$ in $G$ is $\equiv 1 \pmod{p}$.
    \end{enumerate}
  \item \label{rem3} (see \cite[p.~81]{Rotman1995}) Let $P$ be a Sylow $p$-subgroup of a finite group $G.$ Then, $\mathrm{N}_G(H)=H$ for any subgroup $H$ with $\mathrm{N}_G(P)\leq H\leq G.$
  \item \label{rem4} (see \cite[p.~24]{Gorenstein1980}) In a finite solvable group $G$, the factors of every chief series are elementary abelian of prime power order. In particular, every minimal normal subgroup of $G$ is elementary abelian.
  \item \label{rem5} (see \cite[p.~231]{Gorenstein1980}) If $G$ is a finite solvable group, then any $\pi$-subgroup is contained in a Hall $\pi$-subgroup. Moreover, any two Hall $\pi$-subgroup are conjugate.
  \item \label{rem6} (see \cite[p.~197]{Rotman1995}) Let $G$ be a finite group and $p$ be the smallest prime divisor of $|G|$. If a Sylow $p$-subgroup $P$ of $G$ is cyclic, then there is a normal subgroup $N$ of $G$ such that $P\cap N=1$ and $G=PN.$ 
  \end{enumerate}
\end{remarks}

\subsection*{Nilpotent groups}
\label{sec:nilpotent-groups}

Obviously, if $G$ is $\Gamma$-free, then any subgroup of $G$ is also $\Gamma$-free. With this simple remark and by using the fundamental theorem of finite abelian groups, it is not difficult to determine $K_{3,3}$-free abelian groups.

\begin{lem}\label{lem:abelian}
A finite abelian group is $K_{3,3}$-free if and only if it is isomorphic to one
of the following groups $$ \mathbb{Z}_{p^{i}}{\ }(0\leq i\leq 6), \qquad
\mathbb{Z}_{p^{3}}\times \mathbb{Z}_{q}, \qquad \mathbb{Z}_{p^{2}}\times
\mathbb{Z}_{q}, \qquad \mathbb{Z}_{p}\times \mathbb{Z}_{q}, \qquad
\mathbb{Z}_{9}\times \mathbb{Z}_{3}, \qquad \mathbb{Z}_{4}\times
\mathbb{Z}_{2},$$ $$ \mathbb{Z}_{p}\times \mathbb{Z}_{p},\qquad \qquad
\mathbb{Z}_{p}\times \mathbb{Z}_{q}\times \mathbb{Z}_{r}, \qquad  \qquad
\mathbb{Z}_{2}\times \mathbb{Z}_{2}\times \mathbb{Z}_{p}{\ }(p\neq 2)$$ where $p$, $q$, and $r$ are distinct primes.
\end{lem}

\begin{proof}
Let $G$ be a finite abelian group and $p,q$, and $r$ be distinct prime numbers. 

\emph{Case I:} $G$ is a cyclic group. Then there is exactly one subgroup of $G$ of order $n$ for each divisor $n$ of $|G|$. Observe that $\Gamma(G)$ is $K_{3,3}$-free if $G$ is of order $p^i$ ($0\leq i \leq 6$),  $p^2q$, or $pq$; as the number of proper non-trivial subgroups of $G$ would be less than six in such cases. 

\emph{Case I (a):} $|G|=p^i$ ($i>6$). Then $\Gamma(G)$ contains a $K_6$ and cannot be $K_{3,3}$-free. 

\emph{Case I (b):} $G\cong\mathbb{Z}_{p^3q}$. Let $a$ and $b$ be two elements of $G$ of order $p^3$ and of order $q$ respectively. There are exactly six proper non-trivial subgroups of $G$ in this case and five of them, namely $\langle a^{p^2}\rangle$, $\langle a^p\rangle$, $\langle a\rangle$, $\langle a^{p^2},b\rangle$, and $\langle a^p,b\rangle$ form a complete graph in $\Gamma(G)$ as all of them contain $\langle a^{p^2}\rangle$. The remaining vertex $\langle b\rangle$ has degree two in the intersection graph and, therefore, $\Gamma(G)$ is $K_{3,3}$-free. Moreover, since $\Gamma(G)$ contains a $K_5$, it cannot be a proper subgroup of a $K_{3,3}$-free group. (Notice that in a larger group $G$ becomes a vertex and so vertices containing $\langle a^{p^2}\rangle$ form a subgraph containing a $K_6$.) 

\emph{Case I (c):} $G\cong\mathbb{Z}_{p^2q^2}$. Let $a$ and $b$ be two elements of $G$ of order $p^2$ and of order $q^2$ respectively. As in the previous case we have five subgroups forming a $K_5$, namely $\langle a^p\rangle$, $\langle a\rangle$, $\langle a^p,b^q\rangle$, $\langle a,b^q\rangle$, and $\langle a^p,b\rangle$. However, unlike the previous case, the subgroup $\langle b^q\rangle$ is linked by an edge with those last three subgroups forming the $K_5$. Therefore, $\Gamma(G)$ is not $K_{3,3}$-free in this case. 

\emph{Case I (d):} $G\cong\mathbb{Z}_{pqr}$. Let $a,b$, and $c$ be three elements of orders $p,q$, and $r$ respectively. The vertices of $\Gamma(G)$ are $\langle a\rangle$, $\langle b\rangle$, $\langle c\rangle$, $\langle a,b\rangle$, $\langle a,c\rangle$, and $\langle b,c\rangle$. As the valency of $\langle a\rangle$ is two, $\Gamma(G)$ is $K_{3,3}$-free in this case. However, if $G$ is a proper subgroup of a larger group, then maximal subgroups of $G$ together with $\langle a\rangle$, $\langle b\rangle$, and $G$ form a subgraph containing a $K_{3,3}$. 

To sum up, the only possible values for the order of a $K_{3,3}$-free cyclic group are $$ p^i\; (0\leq i\leq 6),\qquad p^3q,\qquad p^2q,\qquad pqr,\qquad pq.  $$ 

\emph{Case II:} $G$ is \emph{not} a cyclic group. Let us make a useful observation. If $G$ is a $K_{3,3}$-free abelian group of order $n$, then $n$ must be one of the above values. This is because for any pair of subgroups $A<B$ of $\mathbb{Z}_n$, there are corresponding subgroups $H<K$ of $G$ such that $|A|=|H|$ and $|B|=|K|$. 

\emph{Case II (a):} $G\cong\mathbb{Z}_p\times\mathbb{Z}_p\times\mathbb{Z}_p$. Observe that maximal subgroups of $G$ are dimension $2$ subspaces of $G$ considering it as a vector space over $\mathbb{F}_p$. Then, by a counting argument the number of maximal subgroups of $G$ is $[(p^3-1)(p^3-p)]/ [(p^2-1)(p^2-p)]=p^2+p+1$. However, by the product formula (see Remark~\ref{rem1}), any pair of maximal subgroups intersects non-trivially; and hence, they form a complete graph in $\Gamma(G)$. Thus, any group containing an elementary abelian subgroup of rank three cannot be $K_{3,3}$-free.

\emph{Case II (b):} $G\cong\mathbb{Z}_{p^2}\times\mathbb{Z}_p$. Let $a$ and $b$ be two generator elements for $G$ of order $p^2$ and of order $p$ respectively. Then, subgroups $\langle a^p\rangle$, $\langle a^p,b\rangle$, $\langle a\rangle$, $\langle ab\rangle,\dots,  \langle ab^{p-1}\rangle$ form a $K_{p+2}$ in $\Gamma(G)$. Hence, the only possible values of $p$ are $2$ and $3$, if $G$ is $K_{3,3}$-free. Actually, $G$ is $K_{3,3}$-free for those primes. Intersection graph of $\mathbb{Z}_9\times\mathbb{Z}_3$ is depicted in Figure~\ref{fig:Z9xZ3}. 

\emph{Case II (c):} $G$ is an abelian $p$-group which is not considered in Cases II (a) and (b). If $G\cong \mathbb{Z}_p\times\mathbb{Z}_p$ then $\Gamma(G)$ is $K_{3,3}$-free, since the intersection graph of a group of order $p^2$ consists of isolated vertices. Otherwise, $G$ contains a subgroup $H$ properly which is isomorphic to $\mathbb{Z}_{p^2}\times\mathbb{Z}_p$ with $p\in\{2,3\}$ by the previous cases. Let $a$ and $b$ be two generator elements for $H$ of order $p^2$ and of order $p$ respectively. Since $H$ is a proper subgroup of $G$, there exists an element $c\in G$ that does not lie in $H$. Now, if $c^p\in\langle a\rangle$, then $\langle a^p\rangle$, $\langle a^p,b\rangle$, $\langle a\rangle$, $\langle ab\rangle$, $\langle c\rangle$, and $H$ form a $K_6$ in $\Gamma(G)$. And if $c^p\notin\langle a\rangle$, then $\langle a^p\rangle$, $\langle a^p,b\rangle$, $\langle a\rangle$, $\langle ab\rangle$, $\langle a^p,c\rangle$, and $H$ form a $K_6$ in $\Gamma(G)$. 

\emph{Case II (d):} $|G|=p^3q$. Since $G$ is a non-cyclic abelian group, either $G\cong(\mathbb{Z}_p\times\mathbb{Z}_p\times\mathbb{Z}_p)\times\mathbb{Z}_q$ or $G\cong(\mathbb{Z}_{p^2}\times\mathbb{Z}_p)\times\mathbb{Z}_q$. However, the first case cannot occur if $G$ is $K_{3,3}$-free in virtue of Case II (a); and in the latter case $\langle a^p\rangle$, $\langle a^p,b\rangle$, $\langle a\rangle$, $\langle ab\rangle$, $\langle a,b\rangle$, and $\langle a^p,c\rangle$ form a $K_6$ in the intersection graph where $a,b$, and $c$ are some generators of $G$ of orders $p^2,p$, and $q$ respectively.

\emph{Case II (e):} $G\cong(\mathbb{Z}_p\times\mathbb{Z}_p)\times\mathbb{Z}_q$. Let $a$ and $b$ be two elements of $G$ generating a subgroup of order $p^2$, and let $c$ be an element of $G$ of order $q$. As in Case II (b), subgroups $\langle c\rangle$, $\langle a,c\rangle$, $\langle ab,c\rangle,\dots, \langle ab^{p-1},c\rangle$, and $\langle b,c\rangle$ form a $K_{p+2}$ in $\Gamma(G)$, and therefore, if $G$ is $K_{3,3}$-free then either $p=2$ or $p=3$. If $p=2$, then $G$ is $K_{3,3}$-free and its intersection graph presented in Figure~\ref{fig:Z2xZ2xZp}. However, if $p=3$, then subgroups $\langle c\rangle$, $\langle a,c\rangle$, $\langle ab,c\rangle$, $\langle ab^2,c\rangle$, $\langle b,c\rangle$ together with $\langle a,b\rangle$ form a subgraph containing $K_{3,3}$.

As abelian groups of order $pqr$ and of order $pq$ are necessarily cyclic, the proof is completed. 
\end{proof}

\begin{figure}
\centering
\begin{subfigure}{.5\textwidth}
  \centering
  \includegraphics[width=.8\textwidth]{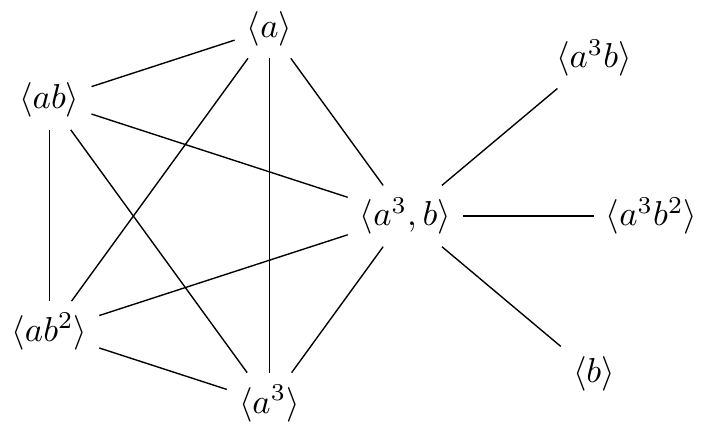}
  \caption{$\Gamma(\mathbb{Z}_9\times\mathbb{Z}_3)$}
  \label{fig:Z9xZ3}
\end{subfigure}%
\begin{subfigure}{.5\textwidth}
  \centering
  \includegraphics[width=.7\textwidth]{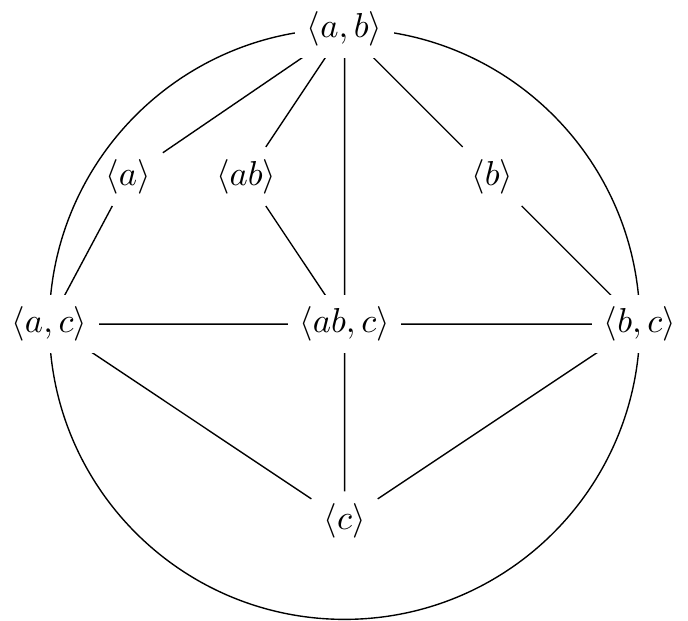}
  \caption{$\Gamma(\mathbb{Z}_2\times\mathbb{Z}_2\times\mathbb{Z}_p)$}
  \label{fig:Z2xZ2xZp}
\end{subfigure}
\caption{}
\label{fig:1}
\end{figure}

Let $G$ be a non-abelian $p$-group of order $p^{\alpha}$ ($\alpha>2$) which is $K_{3,3}$-free. Then, as the quotient of $G$ by the Frattini subgroup $\Phi(G)$ (i.e. the intersection of all maximal subgroups) is elementary abelian, $\Phi(G)$ is a non-trivial subgroup of $G$. That is, there are subgroups $K,L$ of $G$ such that $|K|=p$, $|L|=p^2$, and $K$ is contained by every maximal subgroup of $G$ as well as by $L$. If $p>3$, by Remark~\ref{rem2} there are at least $6$ maximal subgroups each containing a common subgroup, hence $\Gamma(G)$ contains a $K_7$ which is a contradiction. (Notice that for a graph being $K_{3,3}$-free is a more stringent condition than being $K_6$-free.) Thus $G$ is either a $2$-group or a $3$-group. Also, the exponent $\alpha=3$. To see this, suppose that $\alpha= 4$ and consider the case $p=3$. Then there are at least four maximal subgroups of $G$ of order $p^3$ and together with $K$ and $L$, they form a $K_6$ in the intersection graph. If $\alpha=4$, $p=2$ and $|\Phi(G)|=p$, then $G/\Phi(G)$ is elementary abelian of rank $3$ which is not listed in Lemma~\ref{lem:abelian}. Suppose that $|\Phi(G)|=p^2$. If $\Phi(G)$ is cyclic, then the subgroup $Z\leq \Phi(G)$ of order $p$  is in the center of $G$. Since the intersection graph of the quaternion group of order $16$ is $K_9$, we may further assume there are more than one minimal subgroups of $G$. Let $K$ be another minimal subgroup of $G$, then three maximal subgroups together with $Z$, $\Phi(G)$, and $ZK$ form a $K_6$ in $\Gamma(G)$. If $\Phi(G)$ is not cyclic, then we may take three maximal subgroups and the three subgroups of $\Phi(G)$ to form a $K_{3,3}$ in the intersection graph. Finally, $|\Phi(G)|=p^3$ implies $G$ is cyclic which is a contradiction. There are two non-abelian groups of order $8$, namely the dihedral group $D_8$ and the quaternion group $Q_8$; and also there are two non-abelian groups of order $27$, namely $(\mathbb{Z}_3\times \mathbb{Z}_3)\rtimes \mathbb{Z}_3$ and $\mathbb{Z}_9\rtimes \mathbb{Z}_3$. It can be verified that these groups are $K_{3,3}$-free (see Figure~\ref{fig:2} and Figure~\ref{fig:3}). Thus, we almost proved that

\begin{lem}\label{lem:nilpotent}
A finite non-abelian nilpotent group is $K_{3,3}$-free if and only if it is isomorphic to one of the following groups  $$D_8, \qquad\quad Q_8, \qquad\quad (\mathbb{Z}_3\times \mathbb{Z}_3)\rtimes \mathbb{Z}_3, \qquad\quad \mathbb{Z}_9\rtimes \mathbb{Z}_3.$$
\end{lem}

\begin{proof}
  Since a nilpotent group $G$ is the direct product of its Sylow subgroups, at least one of them must be non-abelian. However, $(\mathbb{Z}_3\times \mathbb{Z}_3)\rtimes \mathbb{Z}_3$ and $\mathbb{Z}_9\rtimes \mathbb{Z}_3$ both contains a $K_5$ in their intersections graphs, therefore cannot be a proper subgroup of $G$. Also, if $G$ contains $D_8$ properly, then the three maximal subgroups together with $D_8$ and $\Phi(D_8)$ form a $K_5$ in $\Gamma(G)$. If we take another minimal subgroup $K$ which is not a subgroup of $D_8$, then the subgroup $\Phi(D_8)K$ would be a sixth vertex which is connected by an edge with each vertices of $K_5$. Same argument is valid also for $Q_8$.
\end{proof}

\begin{figure}
\centering
\begin{subfigure}{.5\textwidth}
  \centering
  \includegraphics[width=.7\textwidth]{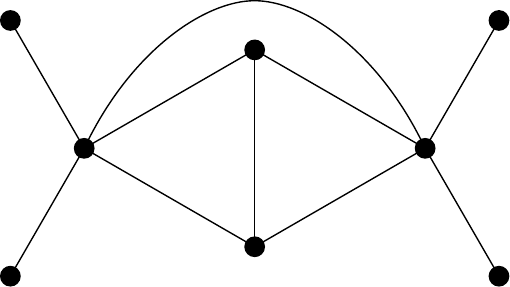}
  \caption{$\Gamma(D_8)$}
  \label{fig:D8}
\end{subfigure}%
\begin{subfigure}{.5\textwidth}
  \centering
  \includegraphics[width=.45\textwidth]{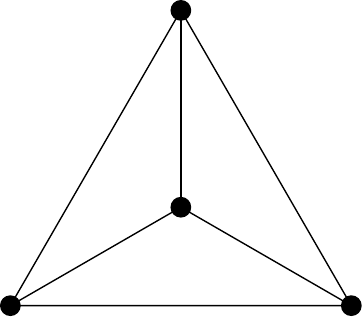}
  \caption{$\Gamma(Q_8)$}
  \label{fig:Q8}
\end{subfigure}
\caption{}
\label{fig:2}
\end{figure}

\begin{figure}
\centering
\begin{subfigure}{.5\textwidth}
  \centering
  \includegraphics[width=.8\textwidth]{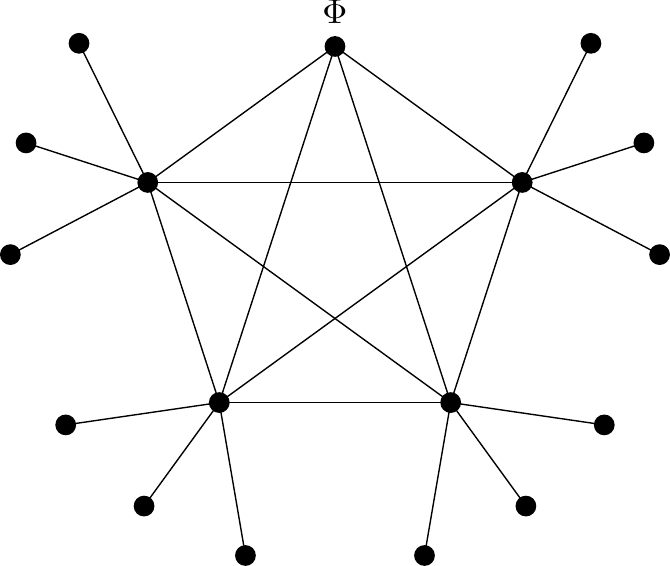}
  \caption{$\Gamma((\mathbb{Z}_3\times\mathbb{Z}_3)\rtimes\mathbb{Z}_3)$}
  \label{fig:Z3xZ3xiZ3}
\end{subfigure}%
\begin{subfigure}{.5\textwidth}
  \centering
  \includegraphics[width=.65\textwidth]{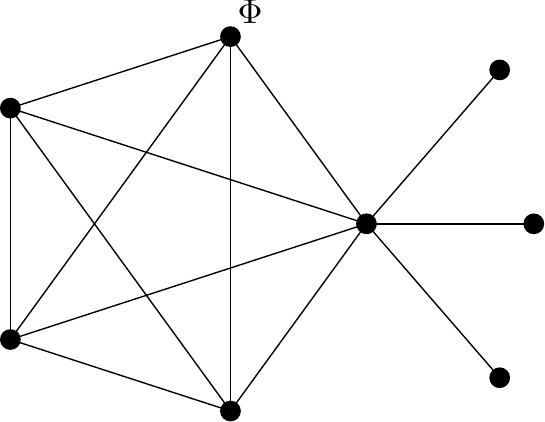}
  \caption{$\Gamma(\mathbb{Z}_9\rtimes\mathbb{Z}_3)$}
  \label{fig:Z9xiZ3}
\end{subfigure}
\caption{Vertices labelled with $\Phi$ represents Frattini subgroups}
\label{fig:3}
\end{figure}

\subsection*{Solvable groups}
\label{sec:solvable-groups}

Let $G$ be a finite non-nilpotent solvable group. Following \cite[Sec.~4]{KayacanYaraneri2015a}, we may reduce the number of cases substantially with regard to the orders of the groups. Let $N$ be a minimal subgroup of $G$. By Remark~\ref{rem4}, $N$ is an elementary abelian group and as a subgroup of $G$ it is $K_{3,3}$-free. Moreover, $N$ is either of rank $1$ or rank $2$ in virtue of Lemma~\ref{lem:abelian}. It is well-known that there is a correspondence between the subgroups of $G$ containing $N$ and the subgroups of $G/N$. Now we make a very useful observation
\begin{itemize}\itemsep0em 
\item the rank of $N$ is $1 \implies \#$ subgroups of $G/N$ is at most $6$;
\item the rank of $N$ is $2 \implies \#$ subgroups of $G/N$ is at most $3$.
\end{itemize}
As a consequence of Sylow and Hall Theorems, the only possible values of $|G/N|$ are $p^i\; (0\leq i\leq5)$, $pq$, and $p^2q$ where $p$ and $q$ distinct prime numbers. Moreover, if $N$ is of rank $2$, then $G/N$ is isomorphic to a cyclic group of prime or prime squared order. Therefore, the only possible cases for the order of $G$ are $$ p^5q, \quad p^4q, \quad p^2qr; \quad p^3q, \quad p^2q, \quad p^2q^2, \quad pqr, \quad pq.$$

In \cite{RajkumarDevi}, the $K_{2,3}$-free groups are determined as a sublist of $K_5$-free groups. Our preceding discussion made it apparent that if $G$ is a $K_{2,3}$-free group and $N\triangleleft G$ is elementary abelian of rank $2$, then either $|G:N|= 1 \text{ or } p$ where $p$ denotes a prime number.

Returning to the possible orders of the non-nilpotent solvable $K_{3,3}$-free groups, we can still eliminate some of the cases by \emph{ad hoc} arguments.

\begin{lem}
  There are no finite non-nilpotent solvable group which is $K_{3,3}$-free and of order $$ p^5q,\qquad\quad p^4q,\qquad \text{ or }\qquad p^2qr $$ where $p,q,r$ distinct prime numbers.
\end{lem}

\begin{proof}
  Let $G$ be a $K_{3,3}$-free group, and let $N$ be a minimal normal subgroup of $G$. First, consider the case $|G|=p^4q$. Clearly, $|N|=q$. Since the number of the subgroups of $G/N$ is at most $6$, $G/N$ is isomorphic to $\mathbb{Z}_{p^4}$. Let $A<B<C<D$ be a chain of non-trivial $p$-subgroups of $G$. Then, one may form $NA$, $NB$ and $NC$ which are proper subgroups. As the orders of those groups are different, they form a $K_7$ in $\Gamma(G)$ (intersection of any two of them contains $A$). Therefore, $G$ is not $K_{3,3}$-free. Similar arguments can also be applied for $|G|=p^5q$ case.

Next, suppose $|G|=p^2qr$. Clearly, $N$ is not a $p$-group. Without loss of generality we may assume that $|N|=r$. Then $|G/N|=p^2q$. If $G/N$ is not cyclic, then the number of subgroups of $G/N$ exceeds $6$. This is clear if Sylow $p$-subgroup of $G/N$ is elementary abelian. And if $G/N$ is not abelian, then there must be a non-normal subgroup (since there is no subgroup of $G/N$ isomorphic to $Q_8$, it is not Hamiltonian) implying there are more than $6$ subgroups. Hence, $G/N$ must be cyclic. In this case there are subgroups of $G/N$ of orders $p^2$, $pq$, $p$, and $q$. Then, by the Correspondence Theorem there are four subgroups containing $N$, say $A$, $B$, $C$, $D$ of orders $p^2r$, $pqr$, $pr$, and $qr$ respectively. Let $T$ be a subgroup of order $p^2q$. By the product formula $T$ intersects $A$, $B$, and $C$ non-trivially. That is, $A, B, C, D, N, T$ span a subgraph in $\Gamma(G)$ containing $K_{3,3}$.
\end{proof}

Now we examine the other cases. 

\begin{lem}\label{lem:p3q}
  Let $ G $ be a non-nilpotent group of order $ p^3q $ where $p$ and $q$ are distinct prime numbers. Then, $ G $ is $K_{3,3}$-free if and only if it is isomorphic to $$\mathbb{Z}_q \rtimes _{\alpha}\mathbb{Z}_{p^3}=\langle a, b \mid a^q = b^{p^3} =1, bab^{-1} = a^{\alpha} \rangle$$ where $p^3$ divides $q-1$ and $\alpha$ is any integer not divisible by $q$ whose order in the unit group $\mathbb{Z}_q^*$ of $\mathbb{Z}_q$ is $p^3$.
\end{lem}

\begin{proof}
Suppose that $G$ is $K_{3,3}$-free. Clearly, the order of the minimal normal subgroup $N$ cannot be $p^2$. Therefore, we only need to consider the following two cases.

\emph{Case I:} $|N|=p$. Since the number of subgroups of $G/N$ is at most $6$, we have $G/N\cong \mathbb{Z}_{p^2q}$. Let $P$ be a Sylow $p$-subgroup of $G$ containing $N$; and let $Q$ be a Sylow $q$-subgroup. By the Correspondence Theorem $P$ is the unique Sylow $p$-subgroup containing $N$, hence it is normal in $G$. As $G$ is not a nilpotent group, $Q$ is not a normal subgroup of $G$. Let $H$ and $K$ be the subgroups containing $N$ with orders $pq$ and $p^2q$ respectively. Since $NQ^g$ is a subgroup of order $pq$ for any conjugate $Q^g$ of $Q$ which contains $N$ and since $H$ is the unique subgroup of order $pq$ containing $N$ by the Correspondence Theorem, $H$ contains all conjugates of $Q$ and this implies the number of Sylow $q$-subgroups of $G$ is $|H:N|=p$. In particular $H\cong \mathbb{Z}_p\rtimes\mathbb{Z}_q$ and $p>q$. Then, the normalizer $\mathrm{N}_G(Q)$ has order $p^2q$. Moreover, since $\mathrm{N}_G(Q)$ is self-normalizing by Remark~\ref{rem3}, there are $p$ conjugates of $\mathrm{N}_G(Q)$ which are different from the normal subgroup $K$. By the product formula, $P,H,K$ and $p$ conjugates of $\mathrm{N}_G(Q)$ pairwise intersect non-trivially. In other words, those subgroups form a $K_{p+3}$ on $\Gamma(G)$. As $p\geq 3$, $G$ cannot be $K_{3,3}$-free in this case.

\emph{Case II:} $|N|=q$. Since $|G:N|=p^3$ and since the number of subgroups of $G/N$ is at most six, $G/N$ is isomorphic to either $\mathbb{Z}_{p^3}$ or $Q_8$. Notice that a group with a unique maximal subgroup is necessarily cyclic and by Remark~\ref{rem1} a non-cyclic $p$-group has at least three maximal subgroups. Therefore, $G/N$ must have a unique minimal subgroup even if it is not cyclic.  

\emph{Case II (a):} $G/N\cong \mathbb{Z}_{p^3}$. Take three non-trivial $p$-subgroups $A<B<C$ and form $NA$ and $NB$. As the orders of those groups are different, they form a $K_5$ in $\Gamma(G)$. Also, since $G$ is not nilpotent, there are more than one Sylow $p$-subgroups of $G$. If $A$ is contained by a Sylow $p$-subgroup $D$ other than $C$, then together with $D$ we have $6$ proper non-trivial subgroups mutually intersecting non-trivially. On the other hand, if any two Sylow $p$-subgroups intersect trivially, then $\Gamma(G)$ is $K_{3,3}$-free. Notice that $NA$ is the unique subgroup of $G$ of order $pq$ and $NB$ is the unique subgroup of $G$ of order $p^2q$. Let $Q=\langle a\rangle$ and $P=\langle b\rangle$. We want to write a presentation for $G$. Since $Q$ is normal, $bab^{-1}=a^{\alpha}$ for some integer $\alpha$ not divisible by $q$. Observe that, $b^kab^{-k}=a^{\alpha^k}$ for any integer $k$. This implies $\alpha^{p^3}\equiv 1 \pmod{q}$, i.e. the order of $\alpha$ in the unit group $\mathbb{Z}_q^*$ divides $p^3$. Moreover, its order is exactly $p^3$, as otherwise, the intersection of some Sylow $p$-subgroups would be non-trivial. Conversely, the group $$\mathbb{Z}_q \rtimes _{\alpha}\mathbb{Z}_{p^3}=\langle a, b \mid a^q = b^{p^3} =1, bab^{-1} = a^{\alpha} \rangle$$ has the subgroup structure described above and it is $K_{3,3}$-free. See Figure~\ref{fig:ZqxiZp3}.   

\emph{Case II (b):} $G/N\cong Q_8$. Then, there are $5$ non-trivial subgroups of a Sylow $p$-subgroup each containing a unique minimal subgroup $A$. Together with $NA$ we have $6$ subgroups forming a $K_6$ in $\Gamma(G)$. Thus, there is no $K_{3,3}$-free group in this case. 
\end{proof}

\begin{figure}[!ht]
  \centering
    \includegraphics[width=0.7\textwidth]{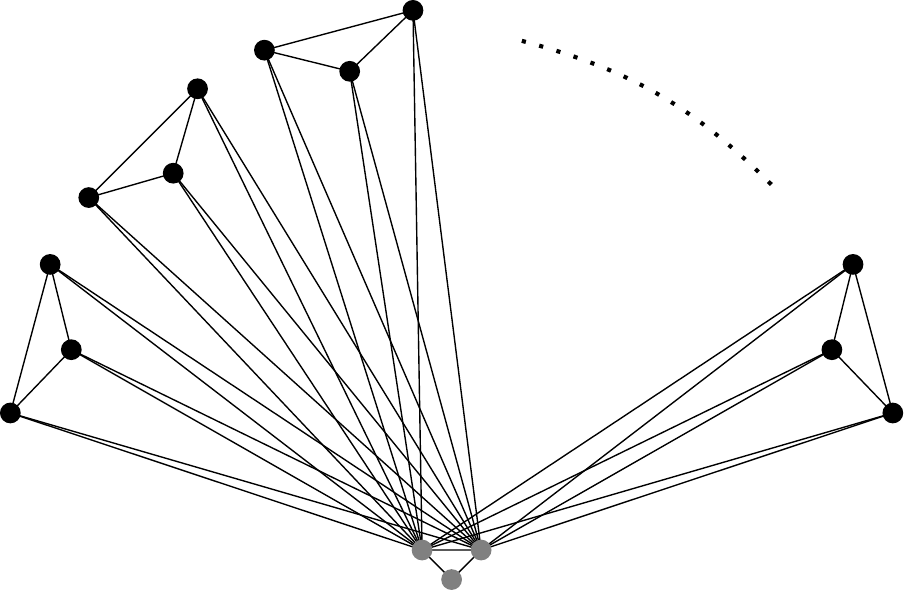}
    \caption{$\Gamma(\mathbb{Z}_q\rtimes_{\alpha}\mathbb{Z}_{p^3})$, gray vertices represents subgroups of orders $q$, $pq$, and $p^2q$}
    \label{fig:ZqxiZp3}
\end{figure}

There are non-nilpotent solvable planar groups of orders 
$$ p^2q, \qquad\quad p^2q^2, \qquad\quad pqr,\qquad \text{ and }\qquad pq $$ which are necessarily $K_{3,3}$-free. Previously, we proved that the groups presented at the second and third items of the following lemma are planar.

\begin{lem}[see {\cite[Lemma~4.6]{KayacanYaraneri2015a}}]\label{lem:p2q}
Let $ G $ be a non-nilpotent group of order $ p^2q $ where $ p $ and $ q $ are distinct prime numbers. Then, $ G $ is $K_{3,3}$-free if and only if it is isomorphic to one of the following groups:
\begin{enumerate}
\item $$\mathbb{Z}_3\rtimes \mathbb{Z}_4,\quad \text{ or } \qquad D_{18},$$
\item $$\mathbb{Z}_q \rtimes _{\alpha}\mathbb{Z}_{p^2}=\langle a, b \mid a^q = b^{p^2} =1, bab^{-1} = a^{\alpha} \rangle$$ where $p^2$ divides $q-1$ and $\alpha$ is any integer not divisible by $q$ whose order in the unit group $\mathbb{Z}_q^*$ of $\mathbb{Z}_q$ is $p^2$, 
\item
$$(\mathbb{Z}_p \times \mathbb{Z}_p) \rtimes _{\beta }\mathbb{Z}_q=\langle a, b, c \mid a^p = b^p = c^q = 1, ab = ba,cac^{-1} = b, cbc^{-1} = a^{-1}b^{\beta} \rangle$$ where $q$ divides $p+1$ and $\beta $ is any integer such that the matrix $ \theta = \left[ \begin{smallmatrix} 0 & -1 \\ 1 & \beta \end{smallmatrix} \right]$ has order $q$ in the group $GL(2,\mathbb{Z}_p)$ and such that $\theta $ has no eigenvalue in $\mathbb{Z}_p.$ 
\end{enumerate}
\end{lem}

\begin{proof}
Suppose that $G$ is $K_{3,3}$-free. There are three possible cases for the order of the minimal normal subgroup $N$ of $G$.

\emph{Case I:} $|N|=p$. Let $P$ be the Sylow $p$-subgroup of $G$, and $Q$ be a Sylow $q$-subgroup of $G$.

\emph{Case I (a):} $P$ is not a normal subgroup of $G$. Then $N$ is contained in every Sylow $p$-subgroups as well as in some subgroups of order $pq$. However, there are $q=1+kp$ conjugates of $P$ and since $G$ is $K_6$-free, we have $p=2$, $q=3$ and $H:=NQ$ must be a normal subgroup of $G$. Notice that, the three Sylow $p$-subgroups together with $N$ and $H$ form a $K_5$ in $\Gamma(G)$. Moreover, $Q$ is a normal subgroup of $G$; otherwise, $H\cong\mathbb{Z}_p\rtimes\mathbb{Z}_q$, however $3\nmid 2-1$. If $P\cong\mathbb{Z}_2\times\mathbb{Z}_2$, there would be a non-normal subgroup $K\cong\mathbb{Z}_2$, as otherwise, $P$ would be a normal subgroup of $G$. Then $QK$ is connected by an edge with two of the three Sylow $p$-subgroups as well as with $H$ which is a contradiction because we assumed that $G$ is $K_{3,3}$-free. Therefore, $P\cong\mathbb{Z}_4$ and we can easily observe that $$\mathbb{Z}_3\rtimes \mathbb{Z}_4=\langle a,b\mid a^3=b^4=1, bab^{-1}=a^2\rangle$$ is $K_{3,3}$-free as it has exactly six proper non-trivial subgroups and the minimal subgroup of order $3$ has degree one in the intersection graph. 

\emph{Case I (b):} $P$ is the normal Sylow $p$-subgroup of $G$. As $G$ is not a nilpotent group by assumption, $Q$ is not a normal subgroup of $G$. 

Suppose that there is a normal subgroup $L$ of $G$ of order $pq$ containing $Q$. Then $L$ contains all conjugates of $Q$, hence $L\cong\mathbb{Z}_p\rtimes\mathbb{Z}_q$ and in particular $q\bigm{|}p-1$. Moreover, by Remark~\ref{rem3} any subgroup containing $\mathrm{N}_G(Q)$ is self-normalizing and this implies $\mathrm{N}_G(Q)\neq Q$, as $L\triangleleft G$ by assumption. However, $\mathrm{N}_G(Q)\neq G$ either, thus $H:=\mathrm{N}_G(Q)$ is of order $pq$ and it is not a normal subgroup of $G$. Let $K$ be the subgroup of $H$ of order $p$. Notice that since $p>q$, we have $K\triangleleft H$. Clearly, conjugates of $H$ together with $K$ and $P$ form a $K_{p+2}$ in $\Gamma(G)$. Therefore $p=3$ and $q=2$. However, any (Sylow) $q$-subgroup is contained by the normal subgroup $L$ implying there is an edge between $L$ and any conjugate of $H$. That is, conjugates of $H$ together with $K$, $P$, and $L$ span a subgraph containing $K_{3,3}$.

Now suppose that there is no normal subgroup of order $pq$. In particular $NQ$ is not a normal subgroup of $G$. As in the previous paragraph, conjugates of $NQ$ together with $N$ and $P$ form a $K_{p+2}$ in $\Gamma(G)$. Therefore $p=3$ and $q=2$, as the number of Sylow $q$-subgroups is $ \equiv 1 \pmod{q}$. (Since any subgroup of index $2$ must be normal, $p\neq 2$.) If $P\cong\mathbb{Z}_3\times\mathbb{Z}_3$, there must be a normal subgroup $K$ of order $p$ different from $N$. To see this, consider the action of $Q$ by conjugation on the set of subgroups of order $p$. (Notice that there are totally four subgroups of order $p$.) Since $N$ is a normal subgroup of order $p$ and the length of an orbit of $Q$ is either $1$ or $2$, there must be a subgroup $K$ fixed by $Q$ and different from $N$. However, $G$ is generated by the elements of $N$, $K$, and $Q$, thus $K$ is a normal subgroup. Then $KQ$ is a group of order $pq$ different from $NQ$ and its conjugates. This is because, $NQ$ and $KQ$ have unique subgroups of order $p$ which are not conjugate to each other. By the product formula any two subgroups of order $pq$ intersects non-trivially. Therefore, conjugates of $NQ$ together with the conjugates of $KQ$ form a $K_6$ in $\Gamma(G)$. Finally, if $P\cong\mathbb{Z}_9$, we have the dihedral group $$ D_{18}=\langle a,b \mid a^9=b^2=1, bab=a^{-1}\rangle $$ which is $K_{3,3}$-free. See Figure~\ref{fig:D18}.

\emph{Case II:} $|N|=q$. As the Sylow $q$-subgroup $N$ is normal and as $G$ is not a nilpotent group, there are at least three Sylow $p$-subgroups, say $P_i$ ($1\leq i\leq q$). 

Suppose that $G/N\cong \mathbb{Z}_p\times\mathbb{Z}_p$. By the Correspondence Theorem there are at least three subgroups $H_j$ ($1\leq j\leq p+1$) of order $pq$ each containing $N$. By the product formula, $P_i\cap H_j\neq 1$ for any $1\leq i,j\leq 3$ and we have six vertices which span a subgraph of $\Gamma(G)$ containing a $K_{3,3}$, contradiction!

Now suppose that $G/N\cong \mathbb{Z}_{p^2}$. If $X=P_i\cap P_j$ is non-trivial for some distinct Sylow $p$-subgroups, then $X$ must be a normal subgroup of $G$ as $\mathrm{N}_G(X)$ contains both $P_i$ and $P_j$. However, this case was considered in Case I (a). If the intersection of any pair of Sylow $p$-subgroups is trivial, then $G$ has a presentation $$\mathbb{Z}_q \rtimes_{\alpha}\mathbb{Z}_{p^2}=\langle a, b \mid a^q = b^{p^2} =1, bab^{-1} = a^{\alpha} \rangle$$ and it is planar. See \cite[Lemma~4.6]{KayacanYaraneri2015a} for details.

\emph{Case III:} $|N|=p^2$. As the Sylow $p$-subgroup $N$ is normal and as $G$ is not a nilpotent group, any subgroup of order $q$ is not normal in $G$. We want to observe that there are no subgroups of $G$ of order $pq$. To see this, first suppose that there is a subgroup $H$ of $G$ of order $pq$. If $H$ is a normal subgroup of $G$, obviously $H$ contains all (Sylow) $q$-subgroups. Then $A=H\cap N$ is normal in $H$ as well as in $G$, since $N$ is abelian and $\langle H,N\rangle=G$. However, this is in contradiction with the assumption that there is no normal subgroup of order $p$. If $H$ is not a normal subgroup of $G$, since a subgroup of smallest prime index must be normal, we have $p>q$. Then, again $A$ is a normal subgroup of $H$ and of $G$ and we have the same contradiction. Therefore, there is no subgroup of $G$ of order $pq$. In that case, $G$ has a presentation $$(\mathbb{Z}_p \times \mathbb{Z}_p) \rtimes _{\beta }\mathbb{Z}_q=\langle a, b, c \mid a^p = b^p = c^q = 1, ab = ba,cac^{-1} = b, cbc^{-1} = a^{-1}b^{\beta} \rangle$$ and it is planar. See \cite[Lemma~4.6]{KayacanYaraneri2015a} for details.
\end{proof}

\begin{figure}[!ht]
  \centering
    \includegraphics[width=0.5\textwidth]{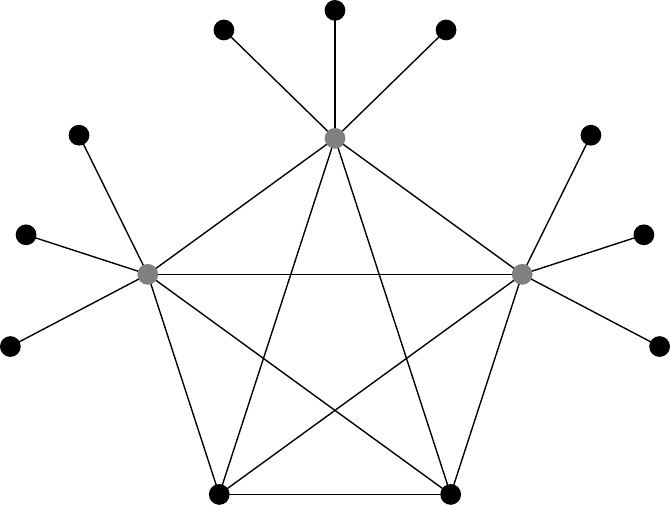}
    \caption{$\Gamma(D_{18})$, gray colored vertices represents subgroups of order $pq$}
    \label{fig:D18}
\end{figure}

\begin{lem}[see {\cite[Lemma~4.7]{KayacanYaraneri2015a}}]\label{lem:p2q2}
  Let $ G $ be a non-nilpotent group of order $ p^2q^2 $ where $p>q$ are distinct prime numbers. Then, $ G $ is $K_{3,3}$-free if and only if it is isomorphic to $$(\mathbb{Z}_p \times \mathbb{Z}_p) \rtimes _{\beta }\mathbb{Z}_{q^2}=\langle a, b, c \mid a^p = b^p = c^{q^2} = 1, ab = ba,cac^{-1} = b, cbc^{-1} = a^{-1}b^{\beta} \rangle$$ where $q^2$ divides $p+1$ and $\beta $ is any integer such that the matrix $ \theta = \left[ \begin{smallmatrix} 0 & -1 \\ 1 & \beta \end{smallmatrix} \right]$ has order $q^2$ in the group $GL(2,\mathbb{Z}_p)$ and such that $\theta ^q$ has no eigenvalue in $\mathbb{Z}_p.$
\end{lem}

\begin{proof} 
Suppose that $ G $ is $K_{3,3}$-free. First we shall observe that the minimal normal subgroup $N$ of $G$ must be a Sylow subgroup. To this end, let us assume $|N|=p$. Then, $G/N\cong \mathbb{Z}_{pq^2}$ and there exists a unique Sylow $p$-subgroup $P$ containing $N$. Since Sylow $p$-subgroups are conjugate and since $N$ is a normal subgroup, $P$ is also a normal subgroup of $G$. Let $Q$ be a Sylow $q$-subgroup. By assumption $G$ is not nilpotent, hence $Q$ is not a normal subgroup of $G$. (Notice that assuming $|N|=q$, one may deduce in a similar fashion that the unique Sylow $q$-subgroup is a normal subgroup of $G$. However, this is not possible for $p>q$.) 

Let $A$, $B$, $C$ be the subgroups of respective orders $pq$, $pq^2$, $p^2q$ containing $N$ and $Q_i$ $(1\leq i\leq 3)$ be three Sylow $q$-subgroups. Let $X$ be a group of order $q$. Since $N$ is a normal subgroup, $NX$ is a group of order $pq$  containing $X$. This implies $A$ contains any group of order $q$, as it is the unique subgroup of order $pq$ containing $N$. That is $A\cap Q_i$ is non-trivial for $1\leq i\leq 3$. This is also true for $B$ and $C$, and so $A,B,C$ together with  $Q_i$ span a subgraph containing a $K_{3,3}$ in $\Gamma(G)$. 

By the preceding discussion we conclude that $N$ is the normal Sylow $p$-subgroup of $G$ and since it is minimal, $N\cong \mathbb{Z}_p\times\mathbb{Z}_p$ by Remark~\ref{rem4}. We know that if the rank of the minimal normal subgroup $N$ is two, then $G/N$ has at most three subgroups, hence it is a cyclic group of prime or prime squared order. Since the order of $G$ is $p^2q^2$, we conclude $Q\cong\mathbb{Z}_{q^2}$ where $Q$ is a Sylow $q$-subgroup. Let $K$ be the unique subgroup of $G$ of order $p^2q$ containing $N$. Since any subgroup of order $p^2q$ contains $N$, we see that $K$ is the unique subgroup of $G$ of order $p^2q$. Moreover, since $NX=K$ for any subgroup $X$ of order $q$, $K$ contains all subgroups of $G$ of order $q$. Also, since any subgroup of a $K_{3,3}$-free group is also $K_{3,3}$-free, $K$ is isomorphic to the third group stated in the previous Lemma~\ref{lem:p2q}. In particular, there is no subgroup of $K$ of order $pq$ and in turn this implies there are no subgroups of $G$ of order $pq^2$ or $pq$. To see this observe that if $H<G$ is of order $pq$, then $H\cap K$ contains a subgroup of order $p$ by the product formula. However, $K$ contains every subgroup of order $q$ which implies $H<K$, contradiction! Similar argument works when $|H|=pq^2$. Hence $G$ is a group with a normal Sylow $p$-subgroup isomorphic to $\mathbb{Z}_p\times\mathbb{Z}_p$ and a non-normal Sylow $q$-subgroup isomorphic to $\mathbb{Z}_{q^2}$ and there are no subgroups of $G$ of order $pq$ or of order $pq^2$. Such a group has a presentation $$(\mathbb{Z}_p \times \mathbb{Z}_p) \rtimes _{\beta }\mathbb{Z}_{q^2}=\langle a, b, c \mid a^p = b^p = c^{q^2} = 1, ab = ba,cac^{-1} = b, cbc^{-1} = a^{-1}b^{\beta} \rangle$$ and it is planar. See \cite[Lemma~4.7]{KayacanYaraneri2015a} for details.
\end{proof}

\begin{lem}[see {\cite[Lemma~4.8]{KayacanYaraneri2015a}}]\label{lem:pqr}
Let $ G $ be a non-nilpotent group of order $ pqr $ where $p<q<r$ are distinct prime numbers. Then, $ G $ is $K_{3,3}$-free if and only if it is isomorphic to  $$\mathbb{Z}_r \rtimes _{\alpha }\mathbb{Z}_{pq}=\langle a, b \mid a^r = b^{pq} =1, bab^{-1} = a^{\alpha} \rangle$$ where $pq$ divides $r-1$ and
  $\alpha$ is any integer not divisible by $r$ whose order in the unit group
  $\mathbb{Z}_r^*$ of $\mathbb{Z}_r$ is $pq$.
\end{lem}

\begin{proof}
Let $R$ be a Sylow $r$-subgroup. Applying Sylow theorems it can be easily observed that $R\triangleleft G$. Since $|G/R|=pq$ and $p<q$, we see that $G/R$ is either abelian (hence cyclic) or isomorphic to $\mathbb{Z}_q\rtimes\mathbb{Z}_p$. In the latter case the number of subgroups of $G/R$ is $q+3$. Since this number must be at most six, $G$ is isomorphic to the symmetric group $S_3\cong\mathbb{Z}_3\rtimes\mathbb{Z}_2$ in the latter case. 

Suppose that $G/R\cong S_3$. By the Correspondence Theorem, there is a unique
subgroup $N$ of order $3r$ and three subgroups $K_i$ ($1\leq i\leq 3$) of order $2r$ containing
$R$. Since $N$ is a Hall $\{3,r\}$-subgroup and $R$ is normal, $N$ is a normal
subgroup of $G$ as well (see Remark~\ref{rem5}). Let $Q$ be the Sylow $3$-subgroup of $G$ contained by
$N$. Then, by Remark~\ref{rem3}, $Q$ is a normal subgroup of $G$. Let $H$ be
a Hall $\{2,3\}$-subgroup of $G$. If $H$ is not a normal subgroup of $G$, then
the number of its conjugates is $|G:H|=r$ and those subgroups together with $Q$
form a $K_{r+1}$ in $\Gamma(G)$. Since $r\geq 5$, the intersection graph cannot
be $K_{3,3}$-free in this case. Also, if $H\triangleleft G$ then it is easy to observe that the subgroups $H$, $N$, $R$ together with $K_i$ ($1\leq i\leq 3$) form a subgraph containing $K_{3,3}$ in the intersection graph.

Suppose that $G/R\cong \mathbb{Z}_{pq}$. By the Correspondence Theorem, there
are unique subgroups $N$ of order $pr$ and $M$ of order $qr$. As in the
preceding paragraph both $M$ and $N$ are normal subgroups. Let $K$ be a
subgroup of order $pq$. Clearly, $K$ is not a normal subgroup of $G$ and in
particular it has $r$ conjugates. Now assume that there exist two distinct
conjugates $K_1$ and $K_2$ of $K$ such that their intersection $X=K_1\cap K_2$
is non-trivial. Then, as $|X|$ is either $p$ or $q$, we have $X\triangleleft
G$; and this implies $X$ is contained by all conjugates of $K$. That is,
conjugates of $K$ together with $X$ form a $K_{r+1}$ in the intersection graph
which is a contradiction as $r\geq 5$. Therefore, any two distinct subgroup of order $pq$ intersects trivially. Such a group has a presentation $$\mathbb{Z}_r \rtimes _{\alpha }\mathbb{Z}_{pq}=\langle a, b \mid a^r = b^{pq} =1, bab^{-1} = a^{\alpha} \rangle$$ and it is planar. See \cite[Lemma~4.8]{KayacanYaraneri2015a} for details.
\end{proof}

Finally, intersection graph of any group of order $pq$ consists of isolated vertices and so $K_{3,3}$-free. For further references we state it as a lemma.

\begin{lem}\label{lem:pq}
Let $G$ be a non-nilpotent group of order $pq$ where $p$ and $q$ are distinct prime numbers and $p>q$. Then, $q\bigm{|} p-1$ and $$G\cong \mathbb{Z}_p\rtimes\mathbb{Z}_q$$ is $K_{3,3}$-free. 
\end{lem}

\subsection*{Non-solvable groups}
\label{sec:non-solvable-groups}

First, we shall show that there is no finite non-abelian simple group which is $K_{3,3}$-free. To this end, we need the following result.

\begin{thm}[{\cite[Theorem~1]{Deskins1961}}]\label{theorem:deskins}
  If the finite group $G$ contains a maximal subgroup $M$ which is nilpotent of class less than $3$, then $G$ is solvable.
\end{thm}

As a consequence of Theorem~\ref{theorem:deskins}, if a Sylow $p$-subgroup $P$ of $G$ is maximal and $|P|=p^3$, then $G$ is solvable. 



\begin{prop}\label{prop:simple}
  If $G$ is a finite minimal simple group, then $\Gamma(G)$ contains a $K_{3,3}$ as a subgraph.
\end{prop}

\begin{proof}
  Consider a finite simple group $U$ which is $K_{3,3}$-free. Then there exists a minimal finite simple group $G$ which is isomorphic to a non-abelian composition factor of some subgroup of $U$. Thus, $G$ must be $K_{3,3}$-free.

Minimal simple groups are known (see \cite[Corollary~1]{Thompson1968}). Thus, $G$ is isomorphic to one of the following groups: $PSL_2(q)$, $Sz(q)$, $PSL_3(3)$.

In view of Feit-Thompson Theorem, $2$ divides $|G|$. Let $S$ be a Sylow $2$-subgroup of $G$. Then $S$ is a $2$-group from Lemmas~\ref{lem:abelian} or \ref{lem:nilpotent}. Thus, either $S \cong \mathbb{Z}_{2^i}$, where $1 \leq i\leq 6$, or $S \in \{\mathbb{Z}_2 \times \mathbb{Z}_2, \mathbb{Z}_4 \times \mathbb{Z}_2, D_8, Q_8\}$.

Since the intersection graphs of each of the $2$-groups $\mathbb{Z}_4 \times \mathbb{Z}_2$, $D_8$ and $Q_8$ contains $K_{1,3}$, those groups must be maximal in $G$. By Theorem~\ref{theorem:deskins}, $G$ is solvable. A contradiction.

Suppose that $S$ is cyclic. Then, by Remark~\ref{rem6}, S has a normal complement in $G$ which contradicts with the assumption that $G$ is simple.

Thus, $S \cong \mathbb{Z}_2 \times \mathbb{Z}_2$. By the Burnside normal complement theorem
(see \cite[Theorem~7.50]{Rotman1995}), the normalizer $N_G(S)$ properly contains $S$. Clearly, $N_G(S)$ is a proper subgroup of $G$, as otherwise, $S$ would be a normal subgroup.

Normalizers of Sylow $2$-subgroups of finite simple groups are known (see \cite[Corollary]{Kondratev2005}). Thus, $G$ is isomorphic to either $PSL_2(q)$, where $q \cong \pm 3 \pmod 8$ (in this case $N_G(S) \cong A_4$), or to $PSL_3(3)$. But $PSL_3(3)$ properly contains $S_4$ (see \cite{Atlas1985}), therefore is not $K_{3,3}$-free.  If $G \cong PSL_2(q)$, where $q \cong \pm 3 \pmod 8$, then there is a subgroup $H \cong D_{q\pm 1}$ of $G$ which is a subgroup of odd index (see \cite[Table~8.7]{BrayEtAl2013}). Take $S$ to be a Sylow $2$-subgroup of $H$. Then $H$, $N_G(S)$, $S$ and three proper subgroups of $S$ form a graph which contains $K_{3,3}$. A contradiction.
\end{proof}

\begin{cor}\label{cor:nonsolvable}
 A finite non-solvable group is not $K_{3,3}$-free.
\end{cor}

\begin{proof} 
Let $G$ be a finite non-solvable group. Since $G$ has a non-abelian simple composition factor which is not $K_{3,3}$-free by Proposition~\ref{prop:simple}, $G$ is not $K_{3,3}$-free as well.
\end{proof}

Our main result follows from Lemmas~\ref{lem:abelian},\,\ref{lem:nilpotent},\,\ref{lem:p3q},\,\ref{lem:p2q},\,\ref{lem:p2q2},\,\ref{lem:pqr},\,\ref{lem:pq} and Corollary~\ref{cor:nonsolvable}.


\end{document}